\documentclass[a4paper, reqno]{amsart}

\usepackage{amsmath}
\usepackage{url}
\usepackage{amsthm}
\usepackage{amsfonts}
\usepackage{appendix}
\usepackage{mathtools}
\usepackage{hyperref}
\usepackage{mathrsfs}
\usepackage{color}
\usepackage{enumerate}
\usepackage{cite}
\usepackage{amssymb}
\usepackage{tikz} % Diagram tool
\usetikzlibrary{calc}

\newtheorem{theorem}{Theorem}[section]

\newtheorem{lemma}[theorem]{Lemma}
\newtheorem{proposition}[theorem]{Proposition}
\newtheorem{corollary}[theorem]{Corollary}

\theoremstyle{definition}
\newtheorem{definition}{Definition}[section]

\theoremstyle{remark}

\numberwithin{equation}{section}

\newcommand{\dist}{\textup{dist}}

\newcommand{\tcap}{\cap\kern-0.7em|\kern0.7em}
\newcommand{\Len}{\textup{Length}}
\newcommand{\inj}{\textup{inj}}

\numberwithin{equation}{section}

\makeatletter
\newcommand*\owedge{\mathpalette\@owedge\relax}
\newcommand*\@owedge[1]{%
  \mathbin{%
    \ooalign{%
      $#1\m@th\bigcirc$\cr
      \hidewidth$#1\m@th\wedge$\hidewidth\cr
    }%
  }%
}
\makeatother

\makeatother
\numberwithin{equation}{section}

\title[Area and boundary length of annuli]{Area and boundary length of surfaces diffeomorphic to annuli}

\author{Tsz-Kiu Aaron Chow}
\address{Department of Mathematics, Columbia University, 2990 Broadway, NY10027, New York, USA}
\email{achow@math.columbia.edu}

\begin{document}

\maketitle

\begin{abstract}
In this paper, we give a proof to a statement in Perelman's paper for finite extinction time of Ricci flow \cite{PL3}. Our proof draws on different techniques from the one given in Morgan-Tian's exposition \cite{MT} and is extrinsic in nature, which relies on the co-area formula instead of the Gauss-Bonnet theorem, and is potentially generalizable to higher dimensions.
\end{abstract}

\section{Introduction}

The qualitative nature of singularities in Ricci flow of dimension 3 was significantly studied in Perelman's three renowned papers \cite{PL1, PL2, PL3}. The arguments in \cite{PL1, PL2, PL3} were detailedly addressed in the expository articles published by Cao-Zhu \cite{CZ}, Kleiner-Lott \cite{KL} and Morgan-Tian \cite{MT}. In the third paper of Perelman \cite{PL3} where he proved the extinction time for Ricci flow is finite using the technique of curve shortening flow, Perelman stated an estimate which concerns the lengths of two curves which are close in the sense that they are the boundary of an annulus with small area. More specifically speaking, given two non-intersecting closed connected curves $c_0$ and $c_1$ such that the disjoint union of them forms the boundary of an annulus $\Sigma$, we quote the statement given by Perelman in \cite{PL3} here:`` If $\varepsilon > 0$ is small then, given any $r>0$, one can find $\bar{\mu}$, depending only on $r$ and on upper bound for sectional curvatures of the ambient space, such that if the length of $c_0$ is at least $r$, each arc of $c_0$ with length $r$ has total curvature at most $\varepsilon$, and $\textup{Area}(\Sigma)\leq\bar{\mu}$, then $\Len(c_1)\geq (1 - 100\varepsilon)\Len(c_0)$."\\

Perelman's statement was addressed by Morgan and Tian in Chapter 18.6 of their exposition \cite{MT}. To our knowledge, this is the only place in the literature where this statement is addressed. In this paper, we will present a proof of Perelman's statement which draws on different techniques. Our proof is extrinsic in nature, and relies on the co-area formula instead of the Gauss-Bonnet theorem. In particular, our proof is potentially generalizable to higher dimensions. To that end, we will first give a slight variant of Perelman's statement. We will first present the full argument to the case where $\mathbb{R}^3$ is the ambient space, as the main argument is easier to follow when it is presented in the $\mathbb{R}^3$ case. After that we will present how the argument can be carried over to the Riemannian case where the ambient space is a fixed Riemannian manifold. The statement corresponding to the $\mathbb{R}^3$ case is given as follows:

\begin{theorem}{\label{main thm}}
		Let $0<\varepsilon<\frac{1}{10000}$ be small. If $\Gamma_0$ and $\Gamma_1$ are two connected closed embedded curves in $\mathbb{R}^3$ such that
	\begin{itemize}
		\item[(i)] $\Gamma_0$ is a $C^1$ curve satisfying $|T(x) - T(y)| \leq\varepsilon$ for all $x, y\in\Gamma_0$ such that $\dist_{\Gamma_0}(x,y)\leq 1$. Here, $T$ denotes the unit tangent vector field along $\Gamma_0$;
		\item[(ii)] $\Len(\Gamma_0) \geq 1$;
		\item[(iii)] $\Gamma_0\sqcup \Gamma_1 = \partial\Sigma$ for a smooth annulus $\Sigma$ with $\textup{Area}(\Sigma)\leq \varepsilon^2$
	\end{itemize}
Then we have
\[\Len(\Gamma_1)\geq (1-C\varepsilon) \Len(\Gamma_0),\]
where $C$ is a positive constant.
\end{theorem}\

Moreover, the statement corresponding to the Riemannian case is given as follows:

\begin{theorem}{\label{main thm M}}
		Let $0<\varepsilon<\frac{1}{10000}$ be small. Let $(M,g)$ be a compact Riemannian manifold of dimension $3$ with $\inj(M,g)\geq 1000$. If $\Gamma_0$ and $\Gamma_1$ are two connected closed embedded curves in $M$ such that
	\begin{itemize}
		\item[(i)] $\Gamma_0$ is a $C^1$ curve satisfying $|P_{y,x}(T(y)) - T(x)| \leq\varepsilon$ for all $x, y\in\Gamma_0$ such that $\dist_{\Gamma_0}(x,y)\leq 1$. Here, $T$ denotes the unit tangent vector field along the curve $\Gamma_0$ and $P_{y,x}: T_yM \to T_xM$ is the parallel transport along $\Gamma_0$;
		\item[(ii)] $\Len(\Gamma_0)\geq 1$;
		\item[(iii)] $\Gamma_0\sqcup \Gamma_1 = \partial\Sigma$ for a smooth annulus $\Sigma\subset M$ with $\textup{Area}(\Sigma)\leq \varepsilon^2$.
	\end{itemize}
Then we have
\[\Len(\Gamma_1)\geq (1-C\varepsilon) \Len(\Gamma_0),\]
where $C$ is a positive constant depending only on the ambient space.
\end{theorem}\

Let us point out here that Theorem \ref{main thm M} implies Perelman's statement in \cite{PL3}, which we summarize here formally:

\begin{corollary}{\label{main cor}}
			Let $0<\varepsilon<\frac{1}{10000}$ be small. Let $(M,g)$ be a compact Riemannian manifold of dimension $3$ and let $K>0$ be an upper bound of the sectional curvatures of $M$. Then given any $r>0$, the following holds:\\
			
			 If $\Gamma_0$ and $\Gamma_1$ are two connected closed embedded curves in $M$ such that
	\begin{itemize}
		\item[(i)] The length of $\Gamma_0$ is at least $r$. And 
			\[\int_I|k|ds\leq \varepsilon\]
			for any sub-segment $I$ of $\Gamma_0$ of length $r$. Here $k$ denotes the curvature of the curve $\Gamma_0$.
			\item[(ii)] $\Gamma_0\sqcup \Gamma_1 = \partial\Sigma$ for a smooth annulus $\Sigma\subset M$ with $\textup{Area}(\Sigma)\leq \frac{1}{1000K}r^2\varepsilon^2$.
	\end{itemize}
Then we have
\[\Len(\Gamma_1)\geq (1-C\varepsilon) \Len(\Gamma_0),\]
where $C$ is a positive constant depending only on the ambient space.
\end{corollary}\

Corollary \ref{main cor} will be proved in Appendix \ref{appendix b}. Now we illustrate the arguments we used to prove Theorem \ref{main thm} and Theorem \ref{main thm M}. The idea is to first construct a one-parameter family of disks around  $\Gamma_0$ such that when the parameter is restricted to a suitable segment, the corresponding disks form a foliation. Using that, we consider the intersection curve of each disk with the annulus $\Sigma$. Since the annulus has small area and the tangent vector field along $\Gamma_0$ cannot change too much along any sub-segment of length 1, one could then use the co-area formula to deduce that most of the intersection curves intersect with $\Gamma_1$ and do not cross with each others. From this, we can deduce the desired length inequality. The main difficulty with this argument is the construction of the desired family of disks, which is due to the fact that we do not have a point-wise control on the curvature of $\Gamma_0$. It could happen that some points on $\Gamma_0$ have arbitrarily high curvature, so that near these points the disks perpendicular to $\Gamma_0$ may have intersection out to any fixed distance. Nevertheless, we could get around this difficulty by first constructing a smooth curve that is $C^1$-close to $\Gamma_0$ and has uniform control on the $C^m$-bound, where $m\geq2$. After that a canonical one-parameter family of disks could be constructed using the smoothed curve in a way that that each disk is perpendicular to the smoothed curve.\\

In Section 2, we will construct a canonical one-parameter family of disks around $\Gamma_0$ when the ambient space is $\mathbb{R}^3$. This will be incorporated into Section 3 to prove Theorem \ref{main thm}. After that we will demonstrate in Appendix \ref{appendix a} how to adapt the arguments in Section 2 and 3 to the Riemannian case, and Theorem \ref{main thm M} will follows accordingly. Lastly, we will give a proof of Corollary \ref{main cor} from Theorem \ref{main thm M} in Appendix \ref{appendix b}.\\

Throughout the paper, a curve is understood to mean a connected curve. And the symbol $C$ is served to denote a positive constant depending only on the ambient space.

\vskip 0.5cm
\noindent\textbf{Acknowledgement:}
The author would like to express his gratitude to his advisor Professor Simon Brendle for his continuing encouragement and many inspiring ideas.

\bigskip

\section{A one-parameter family of disks in $\mathbb{R}^3$}\

By the assumptions in Theorem \ref{main thm}, we immediately have the following lemma:

\begin{lemma}{\label{tangent vector estimate}}
	Let $\gamma_0:\mathbb{R}\to\Gamma_0$ be a parametrization of $\Gamma_0$ by arc-length. Then for $|s - s_0|\leq 100$ we have
	\begin{itemize}
	\item[(i)] 	$|\gamma_0(s) - \gamma_0(s_0) - (s-s_0)\gamma_0'(s_0)| \leq C\varepsilon$;
	\item[(ii)] $|\gamma_0'(s) - \gamma_0'(s_0)|\leq C\varepsilon$.
	\end{itemize}
\end{lemma}\

\begin{corollary}{\label{Gamma length}}
	Under the assumptions of Theorem \ref{main thm}, 
	\[\Len(\Gamma_0) \geq 100.\]
\end{corollary}

\begin{proof}
	Let's denote by $L$ the length of $\Gamma_0$. Let $\gamma_0:\mathbb{R}\to\Gamma_0$ be a parametrization of $\Gamma_0$ by arc-length, then $\gamma_0$ is periodic with period $L$. In particular, we have
	\[ \gamma_0(L) = \gamma_0(0).\]
	Now if $L < 100$, then Lemma \ref{tangent vector estimate} implies $1\leq L\leq C\varepsilon$, which is a contradiction.
\end{proof}\

\begin{corollary}{\label{large length}}
	Suppose that $\hat{\Gamma}$ is a sub-segment in $\Gamma_0$ with end-points $\hat{x}, \hat{y}$.\\ If $|\hat{x} - \hat{y}|\leq 10\varepsilon$ and $\Len(\hat{\Gamma})\geq 1$, then 
	\[\Len(\hat{\Gamma}) \geq 50.\]  	
\end{corollary}

\begin{proof}
	Let's assume in contrary that $\Len(\hat{\Gamma}) < 50$.  Let $\gamma_0:\mathbb{R}\to\Gamma_0$ be a parametrization of $\Gamma_0$ by arc-length and denote by $L$ the length of $\hat{\Gamma}$. Without loss of generality we may assume $\gamma_0(0) = \hat{x}$ and $\gamma_0(L) = \hat{y}$. Then Lemma \ref{tangent vector estimate} implies that 
	\begin{align*}
	C\varepsilon &\geq |\gamma_0(L) - \gamma_0(0) - L\gamma_0'(0)|^2\\
	&\geq L^2 - 2\langle \gamma_0(L) - \gamma_0(0),\ L\gamma_0'(0)\rangle\\
	&\geq L^2 - 2L |\gamma_0(L) - \gamma_0(0)|\\
	&\geq 1 -1000\varepsilon, 
	\end{align*}
	which is a contradiction.
\end{proof}

\bigskip

Next, we would like to construct a one-parameter family of disks which form a smooth foliation when they are restricted to short sub-segments of $\Gamma_0$. The idea is to replace $\Gamma_0$ by a piecewise linear path, and then smooth out the corners using a cut-off function. We first fix a smooth cut-off function $\chi:\mathbb{R}\to\mathbb{R}$ such that:
 \begin{itemize}
 	\item $\chi = 0$ on $(-\infty, -\frac{1}{4}]$;
 	\item $\chi = 1$ on $[\frac{1}{4}, \infty)$.
 \end{itemize}
 We also fix a parametrization $\gamma_0 : \mathbb{R}\to\Gamma_0$ by arc-length,  so that $\gamma_0$ is periodic with period $L = \Len(\Gamma_0)$. By Lemma \ref{tangent vector estimate}, we have $L > 100$. Now we fix a number $k\in [L, 2L]$ and divide $[0, L]$ into sub-intervals of length $\frac{L}{k}\in [\frac{1}{2}, 1]$.
Associated with the curve $\gamma_0$, we define a curve $\tilde{\gamma}_0$ by

\begin{align}{\label{smoothing}}
	\tilde{\gamma}_0(s) &:= \gamma_0\left(\frac{jL}{k}\right) + \left(\frac{ks}{L} - j\right)\chi\left(\frac{ks}{L} - j\right)\left[ \gamma_0\left(\frac{(j+1)L}{k}\right) - \gamma_0\left(\frac{jL}{k}\right)\right]\notag\\
	&\quad\quad\quad + \left( j - \frac{ks}{L} \right)\left( 1 -\chi\left(\frac{ks}{L} - j\right)\right)\left[ \gamma_0\left(\frac{(j-1)L}{k}\right) - \gamma_0\left(\frac{jL}{k}\right)\right],
\end{align}\\
for $s\in \left[\frac{(j - 1/2)L}{k}, \frac{(j + 1/2)L}{k}\right]$, $j = \frac{1}{2}, \frac{3}{2}, .., k -\frac{1}{2}$. 
Then $\tilde{\gamma}_0$ is again periodic with period $L$. 

\begin{lemma}{\label{smoothness}}
	The curve $\tilde{\gamma}_0(s)$ is smooth.
\end{lemma}

\begin{proof}
Observe that for $s\in \left[ \frac{(j - 1/2)L}{k}, \frac{(j - 1/4)L}{k}\right)$, we have

\begin{align*}
	\tilde{\gamma}_0(s) &= 	\gamma_0\left(\frac{jL}{k}\right) + \left(j -\frac{ks}{L} \right)\left[ \gamma_0\left(\frac{(j-1)L}{k}\right) - \gamma_0\left(\frac{jL}{k}\right)\right]\\
	&= \gamma_0\left(\frac{(j-1)L}{k}\right) +  \left(\frac{ks}{L} - j + 1 \right)\left[ \gamma_0\left(\frac{jL}{k}\right) - \gamma_0\left(\frac{(j-1)L}{k}\right)\right].
\end{align*}
Similarly, for $s\in \left[ \frac{(j + 1/4)L}{k}, \frac{(j + 1/2)L}{k}\right)$, we have

\begin{align*}
	\tilde{\gamma}_0(s) &= \gamma_0\left(\frac{jL}{k}\right) +  \left(\frac{ks}{L} - j \right)\left[ \gamma_0\left(\frac{(j+1)L}{k}\right) - \gamma_0\left(\frac{jL}{k}\right)\right].
\end{align*}
From this, we can see that $\tilde{\gamma}_0(s)$ is smooth near $s = \frac{(j - 1/2)L}{k}$ and $s = \frac{(j + 1/2)L}{k}$.

\end{proof}

\begin{lemma}{\label{closeness}}
	For each $s$, we have
	\begin{itemize}
	\item[(i)]	$| \tilde{\gamma}_0(s)-\gamma_0(s)|\leq C\varepsilon$;
	\item[(ii)] $| \tilde{\gamma}_0'(s)-\gamma_0'(s)|\leq C\varepsilon$;
	\item[(iii)] $\Big|\frac{d^m}{ds^m}\tilde{\gamma}_0(s)\Big|\leq C(m)\varepsilon$ for each $m\geq 2$.
	\end{itemize}
\end{lemma}

\begin{proof}
We first fix $j$. Clearly by Lemma \ref{tangent vector estimate}, we have
\begin{equation}{\label{C^0 estimate}}
	\left| \gamma_0(s) - \gamma_0\left(\frac{jL}{k}\right) - \left(s - \frac{jL}{k}\right)\gamma_0'\left(\frac{jL}{k}\right) \right| \leq C\varepsilon 	
\end{equation}
for $s\in \left[\frac{(j-1)L}{k}, \frac{(j+1)L}{k} \right]$. 
In particular, this implies that

\begin{equation}{\label{closeness estimate}}
\begin{cases}
&\left| \gamma_0\left(\frac{(j+1)L}{k}\right) - \gamma_0\left(\frac{jL}{k}\right) -  \frac{L}{k}\gamma_0'\left(\frac{jL}{k}\right) \right| \leq C\varepsilon,\quad\text{and}\\
&\left| \gamma_0\left(\frac{(j-1)L}{k}\right) - \gamma_0\left(\frac{jL}{k}\right) +  \frac{L}{k}\gamma_0'\left(\frac{jL}{k}\right) \right| \leq C\varepsilon
\end{cases}
\end{equation}\\
for  $s\in \left[\frac{(j-1)L}{k}, \frac{(j+1)L}{k} \right]$. Then for assertion (i), we observe that

\begin{align*}
	&\tilde{\gamma}_0(s) - \gamma_0\left(\frac{jL}{k}\right) - \left(s - \frac{jL}{k}\right)\gamma_0'\left(\frac{jL}{k}\right) \\
	&=\left(\frac{ks}{L} - j\right)\chi\left(\frac{ks}{L} - j\right)\left[ \gamma_0\left(\frac{(j+1)L}{k}\right) - \gamma_0\left(\frac{jL}{k}\right)\right]\notag\\
	&\quad\quad\quad + \left( j - \frac{ks}{L} \right)\left( 1 -\chi\left(\frac{ks}{L} - j\right)\right)\left[ \gamma_0\left(\frac{(j-1)L}{k}\right) - \gamma_0\left(\frac{jL}{k}\right)\right]\\
	&\quad\quad\quad - \left(\frac{ks}{L} - j\right)\cdot\frac{L}{k}\gamma_0'\left(\frac{jL}{k}\right)\\
	&=\left(\frac{ks}{L} - j\right)\chi\left(\frac{ks}{L} - j\right)\left[ \gamma_0\left(\frac{(j+1)L}{k}\right) - \gamma_0\left(\frac{jL}{k}\right) - \frac{L}{k}\gamma_0'\left(\frac{jL}{k}\right)\right]\\
	&\quad\quad + \left( j - \frac{ks}{L} \right)\left( 1 -\chi\left(\frac{ks}{L} - j\right)\right)\left[ \gamma_0\left(\frac{(j-1)L}{k}\right) - \gamma_0\left(\frac{jL}{k}\right) + \frac{L}{k}\gamma_0'\left(\frac{jL}{k}\right)\right]
\end{align*}\\
for  $s\in \left[\frac{(j-1/2)L}{k}, \frac{(j+1/2)L}{k} \right]$. Hence, (\ref{closeness estimate}) implies that
\[ 	\left| \tilde{\gamma}_0(s) - \gamma_0\left(\frac{jL}{k}\right) - \left(s - \frac{jL}{k}\right)\gamma_0'\left(\frac{jL}{k}\right) \right| \leq C\varepsilon 	\] 
for  $s\in \left[\frac{(j-1/2)L}{k}, \frac{(j+1/2)L}{k} \right]$. Subsequently assertion (i) follows from combining the above estimate with (\ref{C^0 estimate}). 

\bigskip

Next, we derive that 

\begin{align*}
		&\tilde{\gamma}_0'(s)- \gamma_0'\left(\frac{jL}{k}\right) \\
		&= \left\{\chi\left(\frac{ks}{L} - j\right) + \left(\frac{ks}{L} - j\right)\chi'\left(\frac{ks}{L} - j\right)\right\}\cdot \left\{\frac{k}{L}\left[ \gamma_0\left(\frac{(j+1)L}{k}\right) - \gamma_0\left(\frac{jL}{k}\right)\right] - \gamma_0'\left(\frac{jL}{k}\right)\right\}\\
		&\quad +\left\{\chi\left(\frac{ks}{L} - j\right) + \left(\frac{ks}{L} - j\right)\chi'\left(\frac{ks}{L} - j\right) - 1\right\}\cdot \left\{\frac{k}{L}\left[ \gamma_0\left(\frac{(j-1)L}{k}\right) - \gamma_0\left(\frac{jL}{k}\right)\right] + \gamma_0'\left(\frac{jL}{k}\right)\right\}
\end{align*}\\
for $s\in \left[\frac{(j - 1/2)L}{k}, \frac{(j + 1/2)L}{k}\right]$. Hence by (\ref{closeness estimate}) we obtain
\[ \left|\tilde{\gamma}_0'(s)-\gamma_0'\left(\frac{jL}{k}\right)\right| \leq C\varepsilon.\]
This implies assertion (ii).
\bigskip

For assertion (iii), observe that (\ref{closeness estimate}) gives 

\begin{equation}{\label{C^m estimate}}
	 \left|\left(\gamma_0\left(\frac{(j+1)L}{k}\right) - \gamma_0\left(\frac{jL}{k}\right)\right) -  \left(\gamma_0\left(\frac{jL}{k}\right) - \gamma_0\left(\frac{(j-1)L}{k}\right)\right)\right| \leq C\varepsilon
\end{equation}\\
for $s\in \left[\frac{(j - 1)L}{k}, \frac{(j + 1)L}{k}\right]$.
 On the other hand, we compute that

\begin{align*}
	\tilde{\gamma}_0''(s) &= \frac{k^2}{L^2}	\ \chi'\left(\frac{ks}{L} - j\right) \left(\gamma_0\left(\frac{(j+1)L}{k}\right) - 2\gamma_0\left(\frac{jL}{k}\right) + \gamma_0\left(\frac{(j-1)L}{k}\right)\right) \\
	&\quad + \frac{k^2}{L^2}\left( \frac{ks}{L} - j \right)\ \chi''\left(\frac{ks}{L} - j\right) \left(\gamma_0\left(\frac{(j+1)L}{k}\right) - 2\gamma_0\left(\frac{jL}{k}\right) + \gamma_0\left(\frac{(j-1)L}{k}\right)\right)\\
	&= \left\{ \frac{k^2}{L^2}	\ \chi'\left(\frac{ks}{L} - j\right) + \frac{k^2}{L^2}\left( \frac{ks}{L} - j \right)\ \chi''\left(\frac{ks}{L} - j\right)\right\}\\
	&\quad\quad \times \left\{\gamma_0\left(\frac{(j+1)L}{k}\right) - 2\gamma_0\left(\frac{jL}{k}\right) + \gamma_0\left(\frac{(j-1)L}{k}\right)\right\}
\end{align*}\\
for $s\in \left[\frac{(j - 1/2)L}{k}, \frac{(j + 1/2)L}{k}\right]$. This implies that

\[ \tilde{\gamma}_0^{(m)}(s) = P_m\left(\frac{ks}{L}-j\right) \left\{\gamma_0\left(\frac{(j+1)L}{k}\right) - 2\gamma_0\left(\frac{jL}{k}\right) + \gamma_0\left(\frac{(j-1)L}{k}\right)\right\}\]\\
for $s\in \left[\frac{(j - 1/2)L}{k}, \frac{(j + 1/2)L}{k}\right]$, where $P_m(r) = P_m(r,\chi(r), \chi'(r),\hdots, \chi^{(m)}(r))$ is a polynomial expression in $r, \chi(r),  \chi'(r),\hdots, \chi^{(m)}(r)$. Then by the estimate (\ref{C^m estimate}), we obtain
\[ |\tilde{\gamma}_0^{(m)}(s)| \leq C(m)\varepsilon.\]
\end{proof}

\bigskip

Thus we have a uniform control on higher-order derivative of $\tilde{\gamma}_0$. Next we will use this smoothed curve to construct a one-parameter family of disks. 

\begin{definition} For each $s$, define a disk $D_s$ by
\[ D_s := \{\tilde{\gamma}_0(s) + v:\ \langle \tilde{\gamma}_0'(s), v\rangle = 0,\ |v|< 1\}.\]
\end{definition}

\bigskip

\begin{proposition}{\label{foliated disks}}
If $J\subset\mathbb{R}$ is an interval of length less than 20, then	
\begin{itemize}
\item[(i)] the one-parameter family of disks $\{D_s\}_{s\in J}$ form a smooth foliation around the sub-segment $\tilde{\gamma}_0(J)$;
\item[(ii)] the canonical map of foliation $v_J: \{D_s\}_{s\in J}\to J$ satisfies
	\[ 1-C\varepsilon < |Dv_J| < 1+C\varepsilon.\]	
\end{itemize}
\end{proposition}
\begin{proof}
	
	Let $\{e_1, e_2, e_3\}$ denotes the standard basis of $\mathbb{R}^3$. Without loss of generality, we may assume $0\in J$, $\tilde{\gamma}_0(0) = 0$ and $\tilde{\gamma}_0'(0) = e_3$. By Lemma \ref{closeness}, for $|s|<100$, the map $s\mapsto \tilde{\gamma}_0(s)$ is $C^m$-close to the map $s\mapsto s\cdot e_3$, with error bounded by $O(\varepsilon)$.
	
	For $|s|<100$, $t_1^2 + t_2^2 < 20$, let's consider the map		
	\begin{equation}
		 \Psi (s, t_1, t_2) := \tilde{\gamma}_0(s) + \sum_{i=1}^2 t_i\cdot \left(|\tilde{\gamma}_0'(s)|^2 e_i - \langle\tilde{\gamma}_0'(s), e_i\rangle\tilde{\gamma}_0'(s)\right).
	\end{equation}	
	By Lemma \ref{tangent vector estimate} and Lemma \ref{closeness}, in the region $|s|<100$, $t_1^2 + t_2^2 < 20$, the map $(s, t_1, t_2)\mapsto\Psi(s, t_1, t_2)$ is $C^m$-close to the map $(s, t_1, t_2)\mapsto t_1e_1 + t_2e_2 + s e_3$, with error bounded by $O(\varepsilon)$. Thus by the Inverse Function Theorem, the map $\Psi$ is a diffeomorphism onto its image. Since the map $\Psi$ is $C^m$-close to the map $(s, t_1, t_2)\mapsto t_1e_1 + t_2e_2 + s e_3$, the image of $\Psi$ contains the cylinder
	\[ U = \{t_1e_1 + t_2e_2 + s e_3: |s|< 50,\ t_1^2 + t_2^2 < 10 \}.\]	
	Consequently, we can define a map $v_J: U\to (-50, 50)$ by 
	\[ v_J (x) := \pi_1\circ\Psi^{-1}(x),\]
	where $\pi(s, t_1, t_2) = s$ is the projection onto first factor. Note that the map $v_J$ satisfies
	\[\{D_s\}_{s\in J}\subset U, \quad \text{with}\quad D_s\subset v_J^{-1}(s).\]
	
	Moreover, on the cylinder $U$, the inverse of $\Psi$ is $C^1$ close to the map
	\[ x \mapsto (\langle x, e_3\rangle, \langle x, e_1\rangle, \langle x, e_2\rangle),\]
	with errors bounded by $O(\varepsilon)$. This gives the estimate
	\[ 1-C\varepsilon < |Dv_J| < 1+C\varepsilon.\]	
	Therefore, $v_J$ is a smooth submersion on the cylinder $U$. In particular this implies that the disks $\{D_s\}_{s\in J}$ form a smooth foliation on the cylinder $U$.
\end{proof}

\bigskip

\begin{corollary}{\label{local foliation}}
Associated with each point $x\in \Gamma_0$, there is a disk $\Delta_x$ of radius $1$	such that the followings hold:
\begin{itemize}
\item[(i)] $x\in\Delta_x$;
\item[(ii)] If $\Gamma\subset\Gamma_0$ is a segment of length less than $10$, then the one-parameter family of disks $\{\Delta_x\}_{x\in\Gamma}$ form a smooth foliation around the segment $\Gamma$, and the canonical map of foliation $u_{\Gamma}: \{\Delta_x\}_{x\in\Gamma}\to \Gamma$ satisfies
	\[ 1 - C\varepsilon \leq |Du_{\Gamma}| \leq 1 + C\varepsilon.\] 	
\end{itemize}
\end{corollary}

\begin{proof}
Given a point $x_0$, let $t_0\in [0,L)$ be such that $x_0 = \gamma_0(t_0)$.\\ Since $|\gamma_0(t_0) - \tilde{\gamma}_0(t_0)|\leq \varepsilon$ by Lemma \ref{closeness}, we can find $s_0\in \mathbb{R}$ such that $\gamma_0(t_0)\in D_{s_0}$ and $|s_0 - t_0|\leq 1$. Now we consider the function 
\[ \Phi(t, s) := \langle\tilde{\gamma}_0(s) - \gamma_0(t), \tilde{\gamma}_0'(s)\rangle.\]
At the point $(t_0, s_0)$, we have $\Phi(t_0, s_0) = 0$ and 
\[ \frac{\partial\Phi}{\partial s}(t_0, s_0) = |\tilde{\gamma}_0'(s_0)|^2 + \langle\tilde{\gamma}_0(s_0) - \gamma_0(t_0), \tilde{\gamma}_0''(s_0)\rangle.\]
Since $|\tilde{\gamma}_0(s_0) - \gamma_0(t_0)|\leq 2$ and $|\tilde{\gamma}_0''(s_0)|\leq C\varepsilon$ by Lemma \ref{closeness}, we obtain
\[ 1 - C\varepsilon\leq \frac{\partial\Phi}{\partial s}(t_0, s_0)\leq 1 + C\varepsilon. \]
Subsequently the Implicit Function Theorem implies that there is a unique function $h: I\to\mathbb{R}$ defined on an open interval $I\subset [0,L)$ containing $t_0$ such that $|h(t)-t|\leq 1$ and 
\[ \Phi(t, h(t)) = 0.\]
 Since the point $x_0 = \gamma_0(t_0)$ is arbitrary, the function $h$ can be extended to $[0,L)$ so that there is a unique number $s = h(t)$ for each $t$ such that $|s-t|\leq 1$ and $\gamma_0(t)\in D_s$. Moreover, the function $h$ satisfies 
\begin{equation}{\label{h}}
	1 - C\varepsilon\leq |h'(t)| \leq 1 +C\varepsilon
\end{equation}
by the Implicit Function Theorem. In particular, the above estimate implies that $h$ is a diffeomorphism onto its image by the Inverse Function Theorem. Now, associated with each point $x\in\Gamma_0$, we define
\[ \Delta_x := D_{h\circ\gamma_0|_{[0,L)}^{-1}(x)},\]
where $\gamma_0|_{[0,L)}$ is the restriction of $\gamma_0$ on the interval $[0,L)$.
If $\Gamma\subset\Gamma_0$ is a segment of length 10, we may assume without loss of generality that $\Gamma = \gamma_0(J)$ for some interval $J\subset [0,L)$ of length less than 10. Then Proposition \ref{foliated disks} implies that the one-parameter family of disks $\{\Delta_x\}_{x\in\Gamma}$ coincides with the one-parameter family of disks $\{D_s\}_{s\in h(J)}$, which forms a smooth foliation round the segment $\Gamma$. Moreover, we note that the canonical map of foliation $u_{\Gamma}: \{\Delta_x\}_{x\in\Gamma}\to \Gamma$ is defined by $u_{\Gamma} = \gamma_0\circ h^{-1}\circ v_{h\circ\gamma_0|_{[0,L)}^{-1}(\Gamma)}$. Hence by Proposition \ref{foliated disks} and (\ref{h}) we have
\[ 1 - C\varepsilon \leq |Du_{\Gamma}| \leq 1 + C\varepsilon.\]
	
\end{proof}

\bigskip

\section{Proof of Theorem \ref{main thm}}

\begin{definition}
We define $\mathcal{S}$ to be the collection of all compact segments $\Gamma\subset\Gamma_0$ with the property that there exists a piecewise smooth embedded curve $\alpha\subset\Sigma$ with the following properties:	
\begin{itemize}
\item[(i)] The curve $\alpha$ has the same end-points as $\Gamma$;
\item[(ii)] The union $\alpha\cup\Gamma$ bounds a topological disk in $\Sigma$;
\item[(iii)] $\Len(\alpha)\leq 2\varepsilon$;
\item[(iv)] $\Len(\Gamma)\geq 1$, and the complement $\Gamma_0\setminus\Gamma$ has $\Len(	\Gamma_0\setminus\Gamma)\geq 1$.
\end{itemize}
\end{definition}

\bigskip

\begin{proposition}{\label{empty S}}
	The set $\mathcal{S}$ is empty.
\end{proposition}

\begin{proof} 

Suppose in contrary that $\mathcal{S}$ is nonempty. Then there exists a segment $\hat{\Gamma}\in\mathcal{S}$ with the property that $\Len(\Gamma)\geq\Len(\hat{\Gamma}) - 1$ for all $\Gamma\in\mathcal{S}$. Let $\hat{x}, \hat{y}$ denote the end-points of $\hat{\Gamma}$. By the definition of $\mathcal{S}$, we can find a smooth embedded curve $\hat{\alpha}\subset\Sigma$ with the following properties:

\begin{itemize}
\item[(i)] $\hat{\alpha}$ has end-points $\hat{x}, \hat{y}$;
\item[(ii)] The union $\hat{\alpha}\cup\hat{\Gamma}$ bounds a topological disk in $\Sigma$;
\item[(iii)] $\Len(\hat{\alpha})\leq 2\varepsilon$.
\end{itemize}
Since $\Len(\hat{\Gamma})\geq 1$, Corollary \ref{large length} then implies that $\Len(\hat{\Gamma})\geq 50$.

Let $\hat{x}$ and $\hat{y}$ denote the end-points of $\hat{\Gamma}$. Let $J$ be a segment of $\hat{\Gamma}$ of length 4 such that $2\leq\dist_{\Gamma_0}(x, \hat{x})\leq 10$ for all $x\in J$. By Corollary \ref{local foliation}, there is a one-parameter family of disks $ \{\Delta_x\}_{x\in J}$ which foliates a neighborhood around the segment $J$, and a canonical map of foliation $u_{J}: \{\Delta_x\}_{x\in J}\to J$ near the segment $J$ which satisfies 
\[ u_{J}(y) = x\ \iff\ y\in \Delta_x\]
and 
	\[ 1 - C\varepsilon \leq |Du_{J}| \leq 1 + C\varepsilon.\]
				
		Since $\hat{\alpha}\cup\hat{\Gamma}$ is closed, the disks $\{\Delta_x\}_{x\in J}$ intersect with $\hat{\alpha}\cup\hat{\Gamma}$ in at least two points for almost every $x\in J$. Now we set
		\[J' := \{x\in J: \Delta_x\tcap\hat{\alpha} \neq \emptyset\}.\]
		We apply the co-area formula to give
		
		\begin{align*}
			2\varepsilon &\geq \mathcal{H}^1( \hat{\alpha}\cap\bigcup_{x\in J'}\Delta_x)\\
			&\geq\int_{J'}\int_{\Delta_x\cap\hat{\alpha}}\frac{1}{|D^{\hat{\alpha}}u_J|}d\mathcal{H}^0 d\mathcal{H}^1\\
			&\geq (1-C\varepsilon)\int_{J'}\mathcal{H}^0(\Delta_x\cap\hat{\alpha})d\mathcal{H}^1\\
			&\geq \frac{1}{2}\mathcal{H}^1(J').	
		\end{align*}
		Thus $\mathcal{H}^1(J - J')\geq \mathcal{H}^1(J)-4\varepsilon$.	 By Sard's theorem, we can find a subset $\tilde{J}\subset J - J'$ of full measure such that $\Delta_x$ intersect with $\Sigma$ transversally for every $x\in\tilde{J}$. Thus, we apply the co-area formula again, 		
		\begin{align*}
			\varepsilon^2 &\geq \mathcal{H}^2(\Sigma \cap\bigcup_{x\in \tilde{J}}\Delta_x)	\\
			&\geq \int_{\tilde{J}}\int_{\Delta_x\cap \Sigma}\frac{1}{|D^{\Sigma}u_J|}d\mathcal{H}^1 d\mathcal{H}^1\\
			&\geq (1-C\varepsilon)\int_{\tilde{J}}\mathcal{H}^1(\Delta_x\cap \Sigma)d\mathcal{H}^1\\
			&> (1-C\varepsilon) \mathcal{H}^1(\tilde{J})\cdot\inf_{x\in\tilde{J}}\mathcal{H}^1(\Delta_x\cap \Sigma)\\
			&\geq (1-C\varepsilon)(1-4\varepsilon)\cdot\inf_{x\in\tilde{J}}\mathcal{H}^1(\Delta_x\cap \Sigma).
		\end{align*}
		This gives $\inf_{x\in\tilde{J}}\mathcal{H}^1(\Delta_x\cap \Sigma)\leq \varepsilon$. Therefore, there exists a point $\tilde{x}\in\tilde{J}\subset J$ such that:
		\begin{itemize}
		\item $\Delta_{\tilde{x}}$ does not intersect with $\hat{\alpha}$;
		\item $\Delta_{\tilde{x}}$ intersects with $\Sigma$ transversally;
		\item $\Len(\Delta_{\tilde{x}}\cap\Sigma)\leq 2\varepsilon$.
		\end{itemize}
		Let $\tilde{\alpha}$ denote the connected component of $\Delta_{\tilde{x}}\cap\Sigma$ containing $\tilde{x}$, and let $\tilde{y}$ denote the endpoint of $\tilde{\alpha}$. Then $\Len(\tilde{\alpha})\leq 2\varepsilon$. Since $\tilde{\alpha}$ never intersects $\hat{\alpha}$, it follows that $\tilde{y}\in\hat{\Gamma}$. Moreover, $\tilde{y}\neq\tilde{x}$.
		Now we define $\tilde{\Gamma}$ to be the sub-segment of $\hat{\Gamma}$ with end-points $\tilde{x}$ and $\tilde{y}$. Then $\tilde{\Gamma}\cup\tilde{\alpha}$ bounds a topological disk. Next, since the segment $J$ has length less than 10, thus Corollary \ref{local foliation} implies that the disk $\Delta_{\tilde{x}}$ does not intersect $J$ at any point other than $\tilde{x}$. Consequently we have $\tilde{y}\notin J$ and $\Len(\tilde{\Gamma})\geq 1$. Moreover, since $\tilde{\Gamma}\subset\hat{\Gamma}$ and $\dist_{\Gamma_0}(\tilde{x}, \hat{x})\geq 2$, we have $\Len(\tilde{\Gamma})\leq \Len(\hat{\Gamma}) - 2$. In conclusion, we have shown that there exists a sub-segment $\tilde{\Gamma}$ of $\hat{\Gamma}$ and a smooth embedded curve $\tilde{\alpha}\subset\Sigma$ with the following properties: 
	\begin{itemize}
	\item 	The curve $\tilde{\alpha}$ has the same end-points as $\tilde{\Gamma}$;
	\item The union $\tilde{\alpha}\cup\tilde{\Gamma}$ bounds a topological disk in $\Sigma$;
	\item $\Len(\tilde{\alpha})\leq 2\varepsilon$;
	\item $1\leq \Len(\tilde{\Gamma}) \leq \Len(\hat{\Gamma}) - 2$.
	\end{itemize}
Therefore, we have found a segment $\tilde{\Gamma}\in\mathcal{S}$ with $\Len(\tilde{\Gamma}) \leq \Len(\hat{\Gamma}) - 2$. This contradicts to the definition of $\hat{\Gamma}$.

\end{proof}

\bigskip

\begin{definition} We define $\Lambda$ to be the set of all points $x\in\Gamma_0$ with the following properties:
	\begin{itemize}
		\item[(i)] The disk $\Delta_x$ intersects with $\Sigma,\ \Gamma_0$ and $\Gamma_1$ transversally;
		\item[(ii)] The possibly disconnected curve $\Delta_x\cap\Sigma$ has $\Len(\Delta_x\cap\Sigma)\leq\varepsilon$.
	\end{itemize}
\end{definition}

\bigskip

\begin{lemma}{\label{smallness outside A}}
	Whenever $J\subset\Gamma_0$ is a segment with $1\leq\Len(J)\leq 2$, we have
	\[\mathcal{H}^1(J\cap\Lambda)\geq (1 - 2\varepsilon)\Len(J).\]
\end{lemma}

\begin{proof}
We have $\mathcal{H}^1(\Delta_x\cap\Sigma) > \varepsilon$ for almost every $x\in J\setminus\Lambda$ by definition. Since the canonical map of foliation $u_{J}: \{\Delta_x\}_{x\in J}\to J$ satisfies  $1 - C\varepsilon \leq |Du_{J}| \leq 1 + C\varepsilon$ by Corollary \ref{local foliation}, we can apply the co-area formula:
\begin{align*}
		\varepsilon^2 &\geq \mathcal{H}^2( \Sigma \cap\bigcup_{x\in J\setminus\Lambda}\Delta_x)	\\
		&\geq \int_{J\setminus\Lambda}\int_{\Delta_x\cap\Sigma}\frac{1}{|D^{\Sigma}u_{J}|}d\mathcal{H}^1 d\mathcal{H}^1\\
		&\geq (1-C\varepsilon)\int_{J\setminus\Lambda}\mathcal{H}^1(\Delta_x\cap\Sigma)d\mathcal{H}^1\\
		&> \frac{1}{2}\varepsilon\cdot \mathcal{H}^1(J\setminus\Lambda).
\end{align*}
Hence,
\[ \mathcal{H}^1(J\cap\Lambda)\geq \Len(J) - 2\varepsilon \geq (1 - 2\varepsilon)\Len(J).\]
\end{proof}

\bigskip

\begin{definition} Define a map $\phi:\Lambda\to\partial\Sigma$ as follows:
for each $x\in\Lambda$, define $\phi(x)$ to be the end-point of the connected component of $\Delta_x\cap\Sigma$ that contains $x$.	
\end{definition}

\bigskip

\begin{proposition}{\label{good map}}
	The map $\phi$ satisfies
\begin{itemize}
	\item[(i)] $\phi(\Lambda)\subset \Gamma_1$;
	\item[(ii)] $\phi$ is one-to-one.
\end{itemize}
\end{proposition}

\begin{proof}
To prove $\phi(\Lambda)\subset \Gamma_1$, we suppose in contrary that there exist points $x\in \Lambda$ and $y\in \Gamma_0$ such that $y=\phi(x)$. Let $\Gamma$ denote a segment of $\Gamma_0$ with $x, y$ being its end-points. Then Corollary \ref{local foliation} implies that both $\Gamma$ and $\Gamma_0\setminus\Gamma$ have length not less than 10. Moreover, the assumption $\Len(\Delta_x\cap\Sigma)\leq\varepsilon$ together with the definition of $\phi$  imply the existence of a smooth embedded curve $\alpha$ which has the same end-points as $\Gamma$, $\ \Len(\alpha)\leq\varepsilon$ and bounds a disk in $\Sigma$ with $\Gamma$. This implies that $\Gamma\in \mathcal{S}$, contradicting to the fact that $\mathcal{S} = \emptyset$.

Next, suppose in contrary that there exists two points $x, y\in\Lambda$ and a point $z\in\Gamma_1$ such that $z = \phi(x) = \phi(y)$. Again, we denote by $\Gamma$ a segment of $\Gamma_0$ with $x, y$ being its end-points. Then Corollary \ref{local foliation} implies that both $\Gamma$ and $\Gamma_0\setminus\Gamma$ have length not less than 10. On the other hand, if we let $\alpha, \beta$ to denote the components of $\Delta_x\cap\Sigma, \Delta_y\cap\Sigma$ that contain $x, y$ respectively, then we can find a piecewise smooth embedded curve $\gamma\subset\alpha\cup\beta$ in $\Sigma$ with end-points $x, y$ by concatenating the parts of $\alpha$ and $\beta$ in the following way: Denote by $p$ the point at which $\alpha$ meets with $\beta$ the first time starting from $x$. Let $\alpha'$ to be the sub-segment of $\alpha$ with end-points $x, p$ and let $\beta'$ to be the sub-segment of $\beta$ with end-points $p, y$. Then we can define $\gamma := \alpha'\sqcup\beta'$ to be the desired piecewise smooth embedded curve in $\Sigma$ with end-points being $x, y$. Moreover, we have $\Len(\gamma)\leq 2\varepsilon$. This implies that $\Gamma\in \mathcal{S}$, which contradicts to the fact that $\mathcal{S} = \emptyset$.
\end{proof}

\bigskip

\begin{proof}[Completion of the proof of Theorem \ref{main thm}]\

By Proposition \ref{good map}, we have $\phi(\Lambda)\subset\Gamma_1$ and
	\[\phi(\Lambda)\cap\Delta_x = \{\phi(x)\}\]
	for every $x\in \Lambda$. If $J\subset\Gamma_0$ is a segment with $1\leq\Len(J)\leq 2$, we can apply the co-area formula to $\phi(J\cap\Lambda)$ with the canonical map of foliation $u_{J}$ given by Corollary \ref{local foliation}, we then have
		\begin{align*}
			\mathcal{H}^1(\phi(J\cap\Lambda)) &\geq\int_{J\cap\Lambda}\int_{\Delta_x\cap\phi(J\cap\Lambda)}\frac{1}{|D^{\Gamma_1}u_{J}|}d\mathcal{H}^0 d\mathcal{H}^1\\
			&\geq (1-C\varepsilon)\int_{J\cap\Lambda}\mathcal{H}^0(\Delta_x\cap\phi(J\cap\Lambda))d\mathcal{H}^1\\
			&\geq (1-C\varepsilon)\mathcal{H}^1(J\cap\Lambda).
		\end{align*}
Then Lemma \ref{smallness outside A} subsequently gives
\begin{equation}{\label{phi estimate}}
	\mathcal{H}^1(\phi(J\cap\Lambda))\geq (1 - C\varepsilon)\cdot\Len(J)
\end{equation}
whenever $J\subset\Gamma_0$ is a segment with $1\leq\Len(J)\leq 2$.
Lastly, we divide $\Gamma_0$ into a disjoint union of segments of length between 1 and 2, that is,
\[ \Gamma_0 = \sqcup_{i = 1}^NJ_i,\quad 1\leq\Len(J_i)\leq 2.\]
Then (\ref{phi estimate}) gives
	\begin{align*}
		\Len(\Gamma_1) &\geq \sum_{i=1}^N\mathcal{H}^1(\phi(J_i\cap\Lambda))\\
		&\geq (1-C\varepsilon)\sum_{i=1}^N\Len(J_i)\\
		&= (1-C\varepsilon)\cdot\Len(\Gamma_0).
	\end{align*}
This completes the proof of Theorem \ref{main thm}.
\end{proof}

\bigskip

\appendix

\section{Extension to the Riemannian case}{\label{appendix a}}

In this section, we will discuss how to extend our argument to the Riemannian case. All the hypothesis in Theorem \ref{main thm M} are assumed throughout this section. The obstacle of carrying over the argument in Theorem \ref{main thm} to Theorem \ref{main thm M} is to construct a one-parameter family of disks which form a smooth foliation when they are restricted to short sub-segments of $\Gamma_0$ in the Riemannian case, that is, we seek an analogue of Corollary \ref{local foliation} in the Riemannian case where the ambient space is a Riemannian manifold.\\

 We begin with fixing a parametrization $\gamma_0:\mathbb{R}\to\Gamma_0\subset M$ of $\Gamma_0$ by arc-length so that $\gamma_0$ is periodic with period $L = \Len(\Gamma_0)$. Since $\inj(M,g)\geq 1000$, then for any $|s-s_0|\leq 100$ we may write
 \[ \gamma_0(s) := \exp_{\gamma_0(s_0)}(\xi(s)),\]
 for some unique function $\xi: [s_0 - 100, s_0 + 100]\to T_{\gamma_0(s_0)}M$. With this notation, we have the following analogy of Lemma \ref{tangent vector estimate}:
\bigskip

\begin{lemma}{\label{tangent vector estimate M}}
For $|s-s_0|\leq 100$, the function $\xi(s)$ satisfies
	\begin{itemize}
		\item[(i)]	$|\xi(s) - (s-s_0)\gamma_0'(s_0)|\leq C\varepsilon$;
		\item[(ii)] $|\xi'(s) - \gamma_0'(s_0)| \leq C\varepsilon$.
	\end{itemize}
\end{lemma}

\begin{proof}
We only prove the assertion for $s \in [s_0,s_0+100]$. The case $s \in [s_0-100,s_0]$ works in analogous fashion. Let
 \[\xi(s) = \exp_{\gamma_0(s_0)}^{-1}(\gamma_0(s))\]
  for $s \in [s_0,s_0+100]$. Note that $\xi(s_0)=0$ and $\xi'(s_0) = \gamma_0'(s_0)$. Moreover, since $|\gamma_0'(s)|=1$, we obtain $|\xi'(s)| \leq C$ for all $s \in [s_0,s_0+100]$. Let's consider the following family of geodesics
   \[F(s,t) = \exp_{\gamma_0(s_0)}(t \xi(s))\]
    for $s \in [s_0,s_0+100]$ and $t \in [0,1]$. From the definition of $F$, we clearly have $F(s_0,t) = \gamma_0(s_0)$, $F(s,0) = \gamma_0(s_0)$, and $F(s,1) = \gamma_0(s)$.\\ 
  
  \textbf{\underline{Step 1}:} 
   
    For each $s$, we denote by $X(s,t)$ the parallel transport of the vector $\gamma_0'(s_0)$ along the geodesic $t \mapsto F(s,t)$. Clearly, $X(s_0,t) = \gamma_0'(s_0)$, $X(s,0) = \gamma_0'(s_0)$, and $\nabla_t X(s,t) = 0$. Let $Y(s,t) = \nabla_s X(s,t)$. Clearly, $Y(s,0) = 0$. Moreover,  
    
\begin{align*} 
|\nabla_t Y| 
&= |\nabla_t \nabla_s X| \\ 
&= |\nabla_t \nabla_s X - \nabla_s \nabla_t X| \\ 
&= \left|R\left(\frac{\partial F}{\partial t},\frac{\partial F}{\partial s}\right)X\right| \\ 
&\leq C \left|\frac{\partial F}{\partial t} \wedge \frac{\partial F}{\partial s}\right| |X|.
\end{align*}\\
Since $\frac{\partial F}{\partial t} = (d\exp_{\gamma_0(s_0)})_{t \xi(s)}(\xi(s))$ and $\frac{\partial F}{\partial s} = (d\exp_{\gamma_0(s_0)})_{t \xi(s)}(t\xi'(s))$, we consequently have
\[ |\nabla_t Y| \leq C |\xi(s) \wedge \xi'(s)|.\]
Let us fix $s$ and integrate this inequality over $t \in [0,1]$. This yields 
\[|Y(s,1)| \leq C |\xi(s) \wedge \xi'(s)|.\] 
In other words, 
\[|\nabla_s X(s,1)| \leq C |\xi(s) \wedge \xi'(s)|.\] 
We now integrate this inequality over $s$. This gives 
\[|X(s,1) - P_{s_0,s} (\gamma_0'(s_0))| \leq C \int_{s_0}^s |\xi(r) \wedge \xi'(r)| \, dr,\] 
where $P_{s_0,s}: T_{\gamma_0(s_0)} M \to T_{\gamma_0(s)} M$ denotes the parallel transport along $\gamma_0$. By assumption, $|P_{s_0,s} (\gamma_0'(s_0)) - \gamma_0'(s)| \leq \varepsilon$. Therefore, 
\begin{equation}{\label{parallel vs parallel}}
	|X(s,1) - \gamma_0'(s)| \leq \varepsilon + C \int_{s_0}^s |\xi(r) \wedge \xi'(r)| \, dr.
\end{equation}\\

 \textbf{\underline{Step 2}:} \
  
For each $s$, we denote by $V(s,t)$ the Jacobi field along the geodesic $t \mapsto F(s,t)$ with initial condition $V(s,0) = 0$ and $\nabla_t V(s,t) |_{t=0} = \gamma_0'(s_0)$. We claim the following inequality:
\begin{equation}{\label{parallel vs exp}}
	|V(s,1) - X(s,1)| \leq C |\gamma_0'(s_0) \wedge \xi(s)|.
\end{equation}\\
Indeed, the Jacobi equation implies

\[|\nabla_t\nabla_t V| = \left|R\left(V, \frac{\partial F}{\partial t}\right)\frac{\partial F}{\partial t}\right|\leq C\left|V\wedge \frac{\partial F}{\partial t}\right|\left|\frac{\partial F}{\partial t}\right|.\]
Since  $\frac{\partial F}{\partial t}(s,t) =  (d\exp_{\gamma_0(s_0)})_{t\xi(s)}(\xi(s))$ and $V(s,t) = (d\exp_{\gamma_0(s_0)})_{t\xi(s)}(t\gamma_0'(s_0))$, it follows that
\[|\nabla_t\nabla_t V|\leq C|\gamma_0'(s_0)\wedge \xi(s)|.\]\\
Consequently, the vector field $W(s,t) = V(s,t) - tX(s,t)$ satisfies $W(s,0) = 0$, $\nabla_t W(s,t)|_{t=0} = 0$, and $|\nabla_t \nabla_t W| \leq C|\gamma_0'(s_0)\wedge \xi(s)|$. Integrating this inequality over $t\in [0,1]$ yields $|W(s,1)| \leq C|\gamma_0'(s_0)\wedge \xi(s)|$, which implies (\ref{parallel vs exp}).\\

 \textbf{\underline{Step 3}:} \

Putting (\ref{parallel vs parallel}) and (\ref{parallel vs exp}) together, we obtain 
\[|V(s,1) - \gamma_0'(s)| \leq \varepsilon + C |\gamma_0'(s_0) \wedge \xi(s)| + C \int_{s_0}^s |\xi(r) \wedge \xi'(r)| \, dr.\] 
At this point, observe that $V(s,1) = (d\exp_{\gamma_0(s_0)})_{\xi(s)}(\gamma_0'(s_0))$ and $\gamma_0'(s) = (d\exp_{\gamma_0(s_0)})_{\xi(s)}(\xi'(s))$. Since $\inj(M,g)\geq 1000$, the differential $(d\exp_{\gamma_0(s_0)})_{\xi(s)}: T_{\gamma_0(s_0)} M \to T_{\gamma_0(s)} M$ is invertible and we have a uniform bound for its inverse. Consequently, 
\[|\gamma_0'(s_0) - \xi'(s)| \leq \varepsilon + C |\gamma_0'(s_0) \wedge \xi(s)| + C \int_{s_0}^s |\xi(r) \wedge \xi'(r)| \, dr.\] 
Let $M(s) := |\xi'(s) - \gamma_0'(s_0)|$. Then $|\xi(s) - (s-s_0) \, \gamma_0'(s_0)| \leq \int_{s_0}^s M(r) \, dr$. This implies 
\[|\gamma_0'(s_0) \wedge \xi(s)| = |\gamma_0'(s_0) \wedge (\xi(s) - (s-s_0) \, \gamma_0'(s_0))| \leq C \int_{s_0}^s M(r) \, dr\] 
and 
\begin{align*} 
|\xi(s) \wedge \xi'(s)| 
&\leq |\xi(s) \wedge (\xi'(s) - \gamma_0'(s_0))| + |\xi(s) \wedge \gamma_0'(s_0)| \\ 
&\leq C M(s) + C \int_{s_0}^s M(r) \, dr. 
\end{align*}
Putting everything together, we obtain 
\[M(s) \leq \varepsilon + C \int_{s_0}^s M(r) \, dr\] 
for all $s \in [s_0,s_0+100]$. Consequently Gronwall's inequality gives $M(s) \leq C\varepsilon$ for all $s \in [s_0,s_0+100]$. This implies the lemma.

\end{proof}\

Thus, by using the above lemma we can adapt the proofs of Corollary \ref{Gamma length} and Corollary \ref{large length} to derive the following:

\begin{corollary}{\label{large length M}}\
\begin{itemize}
	\item[(i)]	$\Len(\Gamma_0)\geq 100$.
	\item[(ii)] 	Suppose that $\hat{\Gamma}$ is a segment in $\Gamma_0$ with end-points $\hat{x}, \hat{y}$. If $\dist_M(\hat{x}, \hat{y})\leq 10\varepsilon$ and $\Len(\hat{\Gamma})\geq 1$, then 
	\[\Len(\hat{\Gamma})\geq 50.\]  	
\end{itemize}
\end{corollary}

\bigskip

Next, we will illustrate how to modify the definition of $\tilde{\gamma}_0$ in the Riemannian case.  Similar to Section 2, the idea is to replace $\Gamma_0$ by a broken geodesic, and then smooth out the corners using a cut-off function. we fix the same cut-off function $\chi:\mathbb{R}\to\mathbb{R}$ as in Section 2 and denote $L = \Len(\Gamma_0)$.  
We then fix a number $k\in [L, 2L]$ and divide $[0, L]$ into sub-intervals of length $\frac{L}{k}\in [\frac{1}{2}, 1]$.
Associated with the curve $\gamma_0$, we define a curve $\tilde{\gamma}_0$ by

\begin{align}{\label{smoothing M}}
	\tilde{\gamma}_0(s) &:= \exp_{\gamma_0\left(\frac{jL}{k}\right)} \Bigg\{ \left(\frac{ks}{L} - j\right)\chi\left(\frac{ks}{L} - j\right) \exp_{\gamma_0\left(\frac{jL}{k}\right)}^{-1}\left[\gamma_0\left(\frac{(j+1)L}{k}\right)\right]\notag\\
	&\quad\quad\quad + \left( j - \frac{ks}{L} \right)\left( 1 -\chi\left(\frac{ks}{L} - j\right)\right)\exp_{\gamma_0\left(\frac{jL}{k}\right)}^{-1}\left[ \gamma_0\left(\frac{(j-1)L}{k}\right) \right]\Bigg\},
\end{align}\\
for $s\in \left[\frac{(j - 1/2)L}{k}, \frac{(j + 1/2)L}{k}\right]$, $j = \frac{1}{2}, \frac{3}{2}, .., k -\frac{1}{2}$. 

\bigskip

\begin{lemma}{\label{geodesic}}
 The curve $\tilde{\gamma}_0$ defined by (\ref{smoothing M}) is smooth.	
\end{lemma}

\begin{proof}
Observe that for $s\in \left[ \frac{(j - 1/2)L}{k}, \frac{(j - 1/4)L}{k}\right)$, we have

\begin{align*}
	\tilde{\gamma}_0(s) &= 	\exp_{\gamma_0\left(\frac{jL}{k}\right)}\left\{\left(j -\frac{ks}{L} \right)\cdot\exp_{\gamma_0\left(\frac{jL}{k}\right)}^{-1}\left[ \gamma_0\left(\frac{(j-1)L}{k}\right) \right]\right\}\\
	&= \exp_{\gamma_0\left(\frac{(j-1)L}{k}\right)}\left\{  \left(\frac{ks}{L} - j + 1 \right)\cdot \exp_{\gamma_0\left(\frac{(j-1)L}{k}\right)}^{-1}\left[ \gamma_0\left(\frac{jL}{k}\right) \right]\right\}.
\end{align*}
Similarly, for $s\in \left[ \frac{(j + 1/4)L}{k}, \frac{(j + 1/2)L}{k}\right)$, we have

\begin{align*}
	\tilde{\gamma}_0(s) &= \exp_{\gamma_0\left(\frac{jL}{k}\right)} \left\{  \left(\frac{ks}{L} - j \right)\cdot \exp_{\gamma_0\left(\frac{jL}{k}\right)}^{-1} \left[ \gamma_0\left(\frac{(j+1)L}{k}\right)  \right]\right\}.
\end{align*}
From this, we can see that $\tilde{\gamma}_0(s)$ is smooth near $s = \frac{(j - 1/2)L}{k}$ and $s = \frac{(j + 1/2)L}{k}$.

\end{proof}

\bigskip

For each $j$ and $s\in \left[\frac{(j - 1/2)L}{k}, \frac{(j + 1/2)L}{k}\right]$, we can write
\[ \gamma_0(s) = \exp_{\gamma_0\left(\frac{jL}{k}\right)}(\xi_j(s))\]
and
\[  \tilde{\gamma}_0(s) = \exp_{\gamma_0\left(\frac{jL}{k}\right)}(\tilde{\xi}_j(s)).\]
for some unique functions $\xi_j, \tilde{\xi}_j: \left[\frac{(j - 1/2)L}{k}, \frac{(j + 1/2)L}{k}\right] \to  T_{\gamma_0\left(\frac{jL}{k}\right)}M$.

\bigskip

\begin{lemma}{\label{closeness M}}
	For each $j$ and $s\in \left[\frac{(j - 1/2)L}{k}, \frac{(j + 1/2)L}{k}\right]$, with the notations above we have
\begin{itemize}
	\item[(i)]	$|\tilde{\xi}_j(s) - \xi_j(s)|\leq C\varepsilon$;
	\item[(ii)] $|\tilde{\xi}_j'(s) - \xi_j'(s)|\leq C\varepsilon$;
	\item[(iii)] $\left|\frac{d^m}{ds^m} \tilde{\xi}_j(s)\right|\leq C(m) \varepsilon$ for $m\geq 2$.
	\end{itemize}
\end{lemma}

\begin{proof}
We first fix $j$. Then the assumption implies that there exists a unique path 
\[\xi_j: \left[\frac{(j - 1)L}{k}, \frac{(j + 1)L}{k}\right]\to T_{\gamma_0\left(\frac{jL}{k}\right)}M\]
such that
\[ \gamma_0(s) = \exp_{\gamma_0\left(\frac{jL}{k}\right)}(\xi_j(s))\]
for $s\in \left[\frac{(j - 1)L}{k}, \frac{(j + 1)L}{k}\right]$. 	 By Lemma \ref{tangent vector estimate M} the path $\xi_j(\cdot)$ is $C^1$-close to a line segment, up to errors of order $O(\varepsilon)$. More precisely, we have 	
\begin{align*}
	\begin{cases}
			&\left|\xi_j'(s) -\gamma_0'\left(\frac{jL}{k}\right)\right|\leq C\varepsilon,\quad\text{and}\\
			&\left| \xi_j(s) - \left(s - \frac{jL}{k}\right)\gamma_0'\left(\frac{jL}{k}\right) \right| \leq C\varepsilon	
	\end{cases}
\end{align*}
for $s\in \left[\frac{(j-1) L}{k}, \frac{(j + 1)L}{k}\right]$. In particular, this implies that

\begin{equation}{\label{closeness estimate M}}
\begin{cases}
&\left| \xi_j\left(\frac{(j+1)L}{k}\right)  -  \frac{L}{k}\gamma_0'\left(\frac{jL}{k}\right) \right| \leq C\varepsilon,\quad\text{and}\\
&\left| \xi_j\left(\frac{(j-1)L}{k}\right)  +  \frac{L}{k}\gamma_0'\left(\frac{jL}{k}\right) \right| \leq C\varepsilon
\end{cases}
\end{equation}\\
for  $s\in \left[\frac{(j-1)L}{k}, \frac{(j+1)L}{k} \right]$.  Next, from the definition of $\tilde{\gamma}_0$, we have
\begin{align*}
	\tilde{\xi}_j(s) = \left(\frac{ks}{L} - j\right)\chi\left(\frac{ks}{L} - j\right) \xi_j\left(\frac{(j+1)L}{k}\right) + \left( j - \frac{ks}{L} \right)\left( 1 -\chi\left(\frac{ks}{L} - j\right)\right) \xi_j\left(\frac{(j-1)L}{k}\right)
\end{align*}
for  $s\in \left[\frac{(j-1/2)L}{k}, \frac{(j+1/2)L}{k} \right]$. 
Subsequently, with the estimates (\ref{closeness estimate M}) we may argue as in Lemma \ref{closeness} to obtain \\

\begin{itemize}
	\item[(i)]	$\left|\tilde{\xi}_j(s) - \left(s - \frac{jL}{k}\right)\gamma_0'\left(\frac{jL}{k}\right)\right|\leq C\varepsilon$;
	\item[(ii)] $\left|\tilde{\xi}_j'(s) - \gamma_0'\left(\frac{jL}{k}\right)\right|\leq C\varepsilon$;
	\item[(iii)] $\left|\frac{d^m}{ds^m} \tilde{\xi}_j(s)\right|\leq C(m) \varepsilon$ for $m\geq 2$
\end{itemize}
for  $s\in \left[\frac{(j-1/2)L}{k}, \frac{(j+1/2)L}{k} \right]$. This implies the lemma.
\end{proof}

\bigskip

Once we have Lemma \ref{closeness M}, we can argue as in the later parts of Section 2 to construct a desired one-parameter family of disks. For each $s$, we similarly define a disk by 
\[ D_s := \{\exp_{\tilde{\gamma}_0(s)}(v) :\ \langle \tilde{\gamma}_0'(s), v\rangle = 0,\ |v|< 1\}\]
along $\tilde{\gamma}_0$. Similar to Section 2, such a one-parameter family of disks will form a smooth foliation when the parameter is restricted on intervals of length less than 20 as follows: Consider an arbitrary point on $\tilde{\gamma}_0$, without loss of generality we may assume this point to be $\tilde{\gamma}_0(0)$. Let $\{e_1, e_2, e_3\}$ be a local orthonormal frame around $\tilde{\gamma}_0(0)$ such that $\tilde{\gamma}'(0) = e_3(\tilde{\gamma}_0(0))$. Then, locally around $\tilde{\gamma}_0$, we consider a map $\Psi:\mathbb{R}^3\to M$ which is defined by
 \[\Psi (s, t_1, t_2) := \exp_{\tilde{\gamma}_0(s)}\left\{ \sum_{i=1}^2 t_i\cdot \left(|\tilde{\gamma}_0'(s)|^2 e_i(s) - \langle\tilde{\gamma}_0'(s), e_i(s)\rangle\tilde{\gamma}_0'(s)\right)\right\}.\]
 Then in the region $|s|< 100,\ t_1^2 + t_2^2<20$, the map $(s, t_1, t_2)\mapsto \Psi(s, t_1, t_2)$ is $C^m$-close to the map $(s, t_1, t_2)\mapsto\exp_{\tilde{\gamma}_0(0)}(t_1e_1 + t_2e_2 + se_3)$ by Lemma \ref{closeness M}, , with error bounded by $O(\varepsilon)$. Therefore we can argue as in Proposition \ref{foliated disks} to construct a smooth foliation locally around $|s|<50$. Subsequently Corollary \ref{local foliation} still holds when the ambient space is replaced by a Riemannian manifold. Therefore, the argument in Section 3 can be straightforwardly carried over to the Riemannian case. This proves Theorem \ref{main thm M}.

 \bigskip

 \section{Proof of Corollary \ref{main cor}}{\label{appendix b}}
 
 Observe that if the assumption of Corollary \ref{main cor} holds for $r>0$, then it also holds for any $0<r'<r$. Thus we may assume that $r<1$ without loss of generality. Now, fix some $0<r<1$, we rescale the metric $g$ by a factor of $1000Kr^{-2}$, where $K$ is an upper bound for the sectional curvature. In the rescaled metric $\tilde{g} = 1000Kr^{-2}g$, we have (i) $\textup{Area}_{\tilde{g}}(\Sigma)\leq \varepsilon^2$; (ii) $\textup{Sect}_{\tilde{g}}\leq\frac{1}{1000}$ and hence $\inj(M, \tilde{g})\geq 1000$; and (iii) in the rescalied metric $\Gamma_0$ has length greater than 1 and satisfies $\int_I|k_{\tilde{g}}|ds\leq\varepsilon$ for any sub-segment $I$ of length $1$.
 
 Moreover, the property that $\int_I|k_{\tilde{g}}|ds\leq\varepsilon$ for any sub-segment $I$ of length $1$ implies that $|P_{y,x}(T(y)) - T(x)| \leq\varepsilon$ for all points $x, y\in I$ and any sub-segment $I$ of length $1$. We have recovered the assumptions of Theorem \ref{main thm M}, and therefore Corollary \ref{main cor} is implied by Theorem \ref{main thm M}.

\bibliographystyle{amsplain}
\bibliography{citation}

\end{document}